\documentclass[12pt]{amsart}

\usepackage{amssymb,amsmath,latexsym,amsthm}
\usepackage[english]{babel}

\newtheorem{theorem}{Theorem}[section]
\newtheorem{lemma}{Lemma}[section]
\newtheorem{proposition}{Proposition}[section]
\newtheorem{cor}{Corollary}[section]

\theoremstyle{definition}

\newtheorem*{remark}{Remark}
\newtheorem*{example}{Example}

\begin{document}

\title{Representation fields for orders of small ranks}


\author{\sc Luis Arenas-Carmona}


\newcommand\Q{\mathbb Q}
\newcommand\alge{\mathfrak{A}}
\newcommand\Da{\mathfrak{D}}
\newcommand\Ha{\mathfrak{H}}
\newcommand\oink{\mathcal O}
\newcommand\matrici{\mathbb{M}}
\newcommand\Txi{\lceil}
\newcommand\ad{\mathbb{A}}
\newcommand\enteri{\mathbb Z}
\newcommand\finitum{\mathbb{F}}
\newcommand\bbmatrix[4]{\left(\begin{array}{cc}#1&#2\\#3&#4\end{array}\right)}

\maketitle

\begin{abstract}
A representation field for a non-maximal order $\Ha$ in a central
simple algebra $\alge$ is a subfield of the spinor class field of maximal
orders which determines the set of spinor genera of maximal orders
representing $\Ha$. In our previous work we have proved the existence
 of the representation field for several important families of suborders, like commutative orders,
 while we have also found examples where
the representation field fails to exist.  In this article, we prove that the representation field is defined
for any order $\Ha$ of rank $r\leq7$. The same technique yields the existence of representation fields  for any order
 in an algebra whose semi-simple reduction is commutative. We also construct a rank-8 order whose representation field is not defined.
\end{abstract}

\bigskip
\section{Introduction}

Let $K$ be a global field. Let $\alge$ be a central simple
$K$-algebra ($K$-CSA or CSA when $K$ is clear from the context).
Let $\oink=\oink_{K,S}$ be the ring of $S$ integers in $K$, for some finite set $S$
of places in $K$ containing the archimedean places if any. Finally, let $\Sigma$ be the spinor
class field for the set $\mathbb{O}$ of maximal $S$-orders in $\alge$
as defined in \cite{abelianos}, i.e., $\Sigma/K$ is an
abelian extension that classifies maximal orders of $\alge$ into spinor genera, in
the sense that there exists an explicit map
$$\rho:\mathbb{O}\times\mathbb{O}\rightarrow\mathrm{Gal}(\Sigma/K),$$
with the following properties:
\begin{enumerate}
\item $\Da$ and $\Da'$ are in the same spinor genus if and only if
$\rho(\Da,\Da')=\mathrm{Id}_\Sigma$, \item
$\rho(\Da,\Da'')=\rho(\Da,\Da')\rho(\Da',\Da'')$, for any triple
$(\Da,\Da',\Da'')\in\mathbb{O}^3$,
\end{enumerate}
\cite[\S1]{ab2}. The importance of this concept lies in the fact that spinor genera and conjugacy
classes coincide whenever the group $\alge^*$ has strong aproximation with respect to the set $S$,
e.g., when $S$ is the set of archimedean places on a number field and $\alge_\wp$ is not the real quaternion
division algebra for at least one place $\wp\in S$. This happens in particular, for a number field $K$, when the dimension of $\alge$ is larger
than $4$. In this case, the spinor class field gives much information on the set $\alge^*\backslash\mathbb{O}$ of conjugacy classes
of maximal $S$-orders, e.g., the number of such conjugacy classes is $|\alge^*\backslash\mathbb{O}|=[\Sigma:K]$.
 The set of spinor genera of maximal orders also plays an important role in the description of
a fundamental set for the action, of certain arithmethically interesting subgroups of the projective general linear group $\mathrm{PGL}_2(K)$, on some local Bruhat-Tit trees, when $K$ is a global function field \cite{graphsandsheaves}.
Furthermore, for any suborder $\Ha$ there
exists a lower representation field $F=F_-(\Ha)$, which is the
largest subfield satisfying $$\Ha\subseteq\Da\cap\Da'\implies
\rho(\Da,\Da')|_F=\mathrm{Id}_F,\quad\forall(\Da,\Da')\in\mathbb{O}^2.$$ In \cite{ab2}
we gave an explicit description of this field that is valid for an
arbitrary order in $\alge$. For some families of suborders, the field $F$ has also the
 following property:
 \begin{quote}
Property \textbf{RFD}: \emph{If $\Ha\subseteq\Da$ and $\rho(\Da,\Da')|_F=\mathrm{Id}_F$, then
$\Ha$ is contained in some order in the spinor genus of $\Da'$.}
\end{quote}
If Property \textbf{RFD} holds, we call $F$ the representation field $F(\Ha)$ of $\Ha$ and say
that the representation field for $\Ha$ is defined (or exists). In this case the number of spinor genera representing
$\Ha$ is $[\Sigma:F]$. In
\cite{eichler} we found a  rank-$9$ order in a $9$-dimensional CSA for which the
representation field is not defined. Here we prove an existential result in the opposite direction:
\begin{theorem}\label{uno}
Let $K$ be a global field, let $\alge$ be a $K$-CSA, and let $S$ be a  non-empty finite set
of places in $K$ containing the archimedean places if any. Then
the representation field $F(\Ha)$ is defined for any $S$-order $\Ha\subseteq\alge$ whose rank does
not exceed $7$.\end{theorem}
The existence of the
representation field has been proved for several important families
of orders. In \cite{abelianos}, we proved the existence of the representation field whenever
$\Ha$ is commutative, extending previous results of Chevalley \cite{Chevalley}, Chinburg and Friedman \cite{FriedmannQ}, or Linowitz and Shemanske \cite{lino2}.  The same technique used in the proof of Theorem \ref{uno} yields the following generalization:
\begin{theorem}\label{t2}
Let $K$ be a global field, let $\alge$ be a $K$-CSA, and let $S$ be a  non-empty finite set
of places in $K$ containing the archimedean places if any.
Let $\Ha\subseteq\alge$ be an $S$-order.
If the maximal semisimple quotient of the algebra $L=
K\Ha$ is commutative, then the representation field $F(\Ha)$ is defined. 
\end{theorem}
On the other hand, The bound $7$ in Theorem \ref{uno} is optimal, as shown by next result: 
\begin{theorem}\label{t3}
Let $K$ be a number field whose ideal class group has an element of order $4$.
 Then there exists a quadratic extension $F/K$ and an order $\Ha$
of maximal rank in $L=\matrici_2(F)\subseteq\matrici_4(K)$ for which the representation field is not defined.
\end{theorem}
Note that $\matrici_2(F)$ is identified with a sub-ring of $\matrici_4(K)$ via the natural representation of 
$\matrici_2(F)$ on $F^2\cong K^4$.

\section{Representation fields for representations}

Let $K$ be  a  field, and let $B$ be a finite dimensional central division $K$-algebra. Let $L$ be an arbitrary finite
 dimensional $K$-algebra. By a $B$-representation of $L$ we mean a $K$-algebra homomorphism 
$\phi:L\rightarrow \matrici_n(B)$. For any
such $\phi$, the abelian group $B^n=B\times\cdots\times B$, regarded as the space of column vectors, 
has a natural $(L,B)$-bimodule structure given by the products  $ l\cdot v*b=\phi(l)vb$, for any 
$(l,v,b)\in L\times B^n\times B$. We recall a few well known facts from representation theory:
\begin{itemize}
\item A $B$-subspace $W\subseteq B^n$ is $L$-invariant if and only if it is an $(L,B)$-sub-bimodule.
\item For any chain of $(L,B)$-sub-bimodules $\{0\}=M_0\subseteq M_1\subseteq\cdots\subseteq  M_r=B^n$ there exists
a $B$-basis $\mathcal{S}$ of $B^n$, such that the matrix of $\phi(l)$,  for any element $l\in L$,
  with respect to the basis $\mathcal{S}$, has a block decomposition of the form:
\begin{equation}\label{blocks}
\left(\begin{array}{cccc}a_{11}&a_{12}&\cdots&a_{1r}\\0&a_{22}&\cdots&a_{2r}\\
\vdots&\vdots&\ddots&\vdots\\0&0&\cdots&a_{rr}\end{array}\right),\end{equation} where
each $a_{ij}=a_{ij}(l)$ is a block with $\mathrm{dim}_B(M_i)$ rows and $\mathrm{dim}_B(M_j)$ columns.
The map $l\mapsto a_{ii}(l)$ is the representation corresponding to the bi-module $\tilde{M}_i=M_i/M_{i-1}$. 
\item If $M=B^n$ has no non-trivial proper sub-bimodules, or equivalently, if the representation $\phi$ is irreducible,
then $\phi(L)$ is a simple algebra.
\item Given two representations, $\phi,\psi:L\rightarrow\matrici_n(B)$, there exists $a\in\matrici_n(B)^*$ satisfying $\psi(l)=a\phi(l)a^{-1}$ for every element $l$ in $L$, if and only if the bimodules
defined by $\phi$ and $\psi$ are isomorphic.
\end{itemize}

In the remaining of this section $K$, $\alge$, $\oink$ are as in the introduction. 
 Let $\Pi(K)$ be the set of all places, both archimedean and non-archimedean, in $K$. 
Let $S\subseteq\Pi(K)$ be a finite nonempty set, containing the archimedean places, if any, and let 
$U=\Pi(K)-S$. We refer to $U$ and $S$, respectively, as the set of finite and infinite 
places of $K$. Every definition that follows can be extended to the projective case, where $K$ is a global function
field, $S=\emptyset$, while lattices and orders can be interpreted in a sheaf-theoretical context 
(see the remark at the end of \S\ref{rmk1}). For every place $\wp$ we let $I^\wp$ be the maximal ideal corresponding to $\wp$.
 Note that
$m_\wp=I^\wp_\wp$ is the maximal order of the local ring $\oink_\wp\subseteq K_\wp$.

In all that follows, $\Ha$ is an $S$-order on a finite dimensional $K$-algebra $L$, i.e.,  a lattice
$\Ha\subseteq L$ satisfying $K\Ha=L$. Let $\phi:L\rightarrow\alge$ be a
representation of $L$ in a central simple $K$-algebra $\alge$. Let $L_\ad$ be the adelization of $\alge$, i.e., 
$$L_\ad=\left\{a\in\prod_{\wp\in\Pi(K)}L_\wp\Big| a_\wp\in\Ha_\wp\ \textnormal{for almost all }\wp\in U\right\}.$$
As usual, this definition is independent of the choice of the order $\Ha$ on $L$. 
The rings $\alge_\ad$ and $\ad=K_\ad$ are defined analogously. We adopt the convention that $\Ha_\wp=L_\wp$
for $\wp\in S$, and define $\Ha_\ad=\prod_{\wp\in \Pi(K)}\Ha_\wp$. The definition of $\Da_\ad$, for an order $\Da$ on $\alge$ is analogous. Let
$J_K=\ad^*$.  Note that
$\phi(\Ha)$ is an order in $\alge$, whence we can define the
global spinor image, for any maximal $S$-order
$\Da$ on $\alge$ containing $\phi(\Ha)$, by either of the following equivalent formulas \cite{abelianos}:
$$\begin{array}{rcl}H(\phi,\Da|\Ha)&=&\left\{N(a)|a\in\alge^*_\ad,\, a\phi(\Ha_\ad)a^{-1}\subseteq\Da_\ad\right\},\\
H(\phi,\Da|\Ha)&=&J_K\cap\prod_{\wp\in\Pi(K)} H_\wp(\phi,\Da|\Ha),\end{array}$$
  where $N$ is the reduced norm, and   the local spinor image $H_\wp(\phi,\Da|\Ha)$ is defined by
$$H_\wp(\phi,\Da|\Ha)=\left\{N(a)|a\in\alge_\wp,\, a\phi(\Ha_\wp)a^{-1}\subseteq\Da_\wp\right\}\subseteq K_\wp^*.$$
In all that follows we assume that $\Da$ is maximal.
Note that, when  $\phi(\Ha)$ is contained in a second maximal order $\Da'=a\Da a^{-1}$, then $H(\phi,\Da'|\Ha)=
a^{-1} H(\phi,\Da|\Ha)$, and both sets contain the identity. In particular, both  sets
$H(\phi,\Da'|\Ha)$ and $H(\phi,\Da|\Ha)$ generate the same group
$\Gamma(\phi,\Ha)$. The class field $F=F_-(\phi,\Ha)$
corresponding to $\Gamma(\phi,\Ha)K^*$ is called the lower
representation field. Let $\Sigma$  denote the spinor class field of maximal orders, and let
$\rho:\mathbb{O}^2\rightarrow\mathrm{Gal}(\Sigma/K)$ be the distance function
defined in \cite[\S2]{abelianos}, i.e.,  $\rho(\Da,\Da')=[N(a),\Sigma/K]$, where $\Da'=a \Da a^{-1}$ and
$t\mapsto [t,\Sigma/K]$ is the artin map on ideles. Then $F$ is the smallest subfield
satisfying $$\rho(\Da,\Da')|_F=id\implies \phi(\Ha)\subseteq
a\Da'a^{-1}\textnormal{ for some }a\in\alge.$$
Furthermore, the converse implication holds whenever $H(\phi,\Da|\Ha)K^*$ is a group. In the latter case we say that the representation field is defined (and equals $F$).
Next lemma gives an explicit description of $F_-(\phi,\Ha)$ for any order
$\Ha$ and any representation $\phi$. It follows by replacing $\Ha$ by $\phi(\Ha)$  in \cite[Lem.2.1]{ab2}:

\begin{lemma}\label{l4}  Let $\Ha$ be an order and let
$\phi:\Ha\rightarrow\alge$ be a representation satisfying
$\phi(\Ha)\subseteq\Da$ for some maximal order $\Da$  in the CSA
$\alge$.  For every place $\wp\in U$, with maximal order $m_\wp\subseteq\oink_\wp$ and residue field $\mathbb{K}_\wp$,
we let $J_\wp$ be the only maximal two-sided ideal of
$\Da_\wp$ containing $m_\wp1_{\Da_\wp}$, and let $\mathbb{H}_\wp$ be the
image of $\Ha_\wp$ in $\mathbb{D}_\wp=\Da_\wp/J_\wp$. Let $\mathbb{E}_\wp$
be the center of the ring $\mathbb{D}_\wp$. Let $t_\wp$ be the
greatest common divisor of the dimensions of the irreducible
$\mathbb{E}_\wp$-representations of the $\mathbb{K}_\wp$-algebra
$\mathbb{H}_\wp$. Then the lower representation
$F_-(\phi,\Ha)$ is the maximal subfield $F$, of the spinor class
field $\Sigma$, whose inertia degree $f_\wp(F/K)$ divides
$t_\wp$ for every place $\wp$. Furthermore, if every irreducible
$\mathbb{E}_\wp$-representation of $\mathbb{H}_\wp$ has dimension $t_\wp$, 
for every $\wp$, then the representation field is defined.
\end{lemma}

The algebras  $\mathbb{H}_\wp$  in the previous lemma seem to depend heavily on the maximal order $\Da$, but this is not so.

\begin{lemma}\label{l4b} Let $\Ha$ be an order and let
$\phi:\Ha\rightarrow\alge$ be a faithful representation satisfying
$\phi(\Ha)\subseteq\Da$ for some maximal order $\Da$  in the CSA
$\alge$.  For every place $\wp\in U$, with maximal order $m_\wp\subseteq\oink_\wp$ and residue field $\mathbb{K}_\wp$,
the irreducible representations of the algebra $\mathbb{H}_\wp$ are exactly
the irreducible representations of the algebra
 $\widetilde{\mathbb{H}}_\wp=\Ha_\wp/m_\wp\Ha_\wp$.
\end{lemma}

\begin{proof}
It is immediate that  $\mathbb{H}_\wp$ is a homomorphic image of  $\widetilde{\mathbb{H}}_\wp$. It suffices therefore to note
that any idempotent $T$ in $\widetilde{\mathbb{H}}_\wp$ can be lifted to an idempotent $t$ of $\Ha_\wp$ and $ \pi_\wp^{-1}\phi(t)$ is not integral over $\oink_\wp$,
for any local uniformizing parameter $ \pi_\wp$,
whence the image of $T$  in   $\mathbb{H}_\wp$ is never $0$. \end{proof}

\begin{example}
When $\alge$ is  a quaternion algebra, The residual algebra   $\mathbb{D}_\wp$ is either a quadratic extension
$\mathbb{L}_\wp$, of the residue field  $\mathbb{K}_\wp$, or
a matrix algebra $\matrici_2(\mathbb{K}_\wp)$.
Then, either $\mathbb{E}_\wp\mathbb{H}_\wp$ has a unique two-dimensional  irreducible representation or just one-dimensional representations. In either case the last
 condition in the lemma is satisfied, so the representation field exists for all representations in quaternion algebras. 
\end{example}

For simplicity, one can define an upper representation field $F^-(\phi,\Ha)$ which is the class field of 
$\Delta(\phi,\Da|\Ha)$, the group of ideles $\delta$ satisfying
$$\delta K^*H(\phi,\Da|\Ha)=K^*H(\phi,\Da|\Ha).$$ Note that the representation field
is defined if and only if $F^-(\phi,\Ha)=F_-(\phi,\Ha)$ \cite{spinor}.

\begin{remark}\label{rmk1}
This section can be extended, word-by-word, to order over an $A$-curve  $X$, as defined in
\cite{abelianos},with structure sheaf $\oink=\oink_X$ and field of
rational functions $K=K(X)$.  In other words, all results here apply to $X$-orders for a projective curve $X$ over
a finite field, as defined in \cite{brzezinski87}.  This latter setting is called the projective case in all that follows.
Let $|X|$ be the set of closed points of $X$. Note that $|X|$ can be identified with the set $\Pi(K)$ of places of $K$
defined above. We set $U=|X|$ and $S=\emptyset$ in the projective case, i.e.,there are no infinite places. In this case we define the maximal order $I^\wp$ corresponding to a place $\wp$
 as the one-dimensional lattice on $K$
satisfying $$I^\wp(U)=\left\{\begin{array}{lcr}\left\{f\in \oink(U)\big| |f|_\wp<1\right\}&\textnormal{ if }&\wp\in U\\
\oink(U)&\textnormal{ if }&\wp\notin U\end{array}\right\}.$$ 
An $S$-order $\Ha$ on a finite dimensional $K$-algebra $L$ is an $X$-order, i.e., a locally free
sheaf of $\oink$-algebras, whose generic fiber is $L$, i.e., $L=K\otimes_\oink\Ha$  \cite{brzezinski87}. 
Certainly, strong approximation with respect to $S=\emptyset$ never holds, so spinor genera do not coincide with
conjugacy classes in this context, and in fact the number of conjugacy classes of maximal orders can be infinite,
but two orders $\Da$ and $\Da'$ in the same spinor genus have conjugate rings of sections  $\Da(V)$ and $\Da'(V)$
for all open set $V$ outside a finite set. The image of a maximal order $\Ha$ on $L$ under a representation
$\phi:L\rightarrow\alge$ is the sheaf defined by $\phi(\Ha)(V)=\phi\big(\phi(V)\big)$.
\end{remark}

\section{Proof of Theorem \ref{uno} and Theorem \ref{t2}.}

In all of this section we let $E$ be a finite dimensional algebra
over a local field $k$. We let $\mathfrak{B}$ be a CSA over $k$, 
and we fix a representation $\phi:E\rightarrow \mathfrak{B}$.
 Let $\mathfrak{E}$ be a (local) order in $E$. Recall that $\mathfrak{B}$ is isomorphic to a matrix 
 algebra $\matrici_f(B)$ over a division algebra $B$, so we can assume 
 $\mathfrak{B}=\matrici_f(B)$. In particular, the space of column
vectors $B^f$ is a $(E,B)$-bimodule (\S2). 
Consider a composition series $\{0\}=M_0\subseteq\cdots\cdots\subseteq
M_r=B^f$ of $(E,B)$-bimodules.

\begin{lemma}\label{l3} Let $\mathfrak{E}$ be an order in the $k$-algebra $E$, where $k$ is a local field.
Let $\phi:E\rightarrow\matrici_f(B)$ be a representation, and let 
$\{0\}=M_0\subseteq\cdots\cdots\subseteq M_r=B^f$  be a composition series of the corresponding bimodule.
 Let $C_i$ be
the algebra of right-$B$-linear maps in $M_i/M_{i-1}$, and let 
$\phi_i:L\rightarrow C_i$ denote the natural representation.
Then  there exist
a family of maximal orders $\{\Da_i\}$, with
$\phi_i(\mathfrak{E})\subseteq\Da_i\subseteq C_i$, and a maximal order
$\Da\subseteq\mathfrak{B}$, containing $\phi(\mathfrak{E})$, such that
$$H(\phi,\Da|\mathfrak{E})\supseteq\prod_{i=1}^rH(\phi_i,\Da_i|\mathfrak{E}).$$
\end{lemma}

\begin{proof}
By the remarks at the begining of \S2, there exists a basis of $B^m$ for which the algebra $L=\phi(E)$ is
contained in the ring of matrices with a block decomposition of
the form (\ref{blocks}).
It suffices to see that whenever $u_i$ is a generator for
$\Da_i|\mathfrak{E}$, then
\begin{equation}\label{blocks2}\left(\begin{array}{cccc}\pi^{t_1}u_1&0&\cdots&0\\0&\pi^{t_2}u_2&\cdots&0\\
\vdots&\vdots&\ddots&\vdots\\0&0&\vdots&\pi^{t_r}u_r\end{array}\right),\end{equation}
is a generator for $\Da|\mathfrak{E}$
as soon as each difference $t_i-t_{i-1}$ is chosen big enough, so that the block $\pi^{t_i-t_j}a_{ij}(h)$ is integral for every $i<j$
and every $h\in \mathfrak{E}$. Note that we can achieve this by replacing  $t_i$ by $t_i+fiN$ for large $N$, which does not modify the class of $t_i$ in $\enteri/f\enteri$. The rest of the proof is a word by word trasliteration of the proof of Lemma 4.2 in \cite{abelianos},
and it is therefore omitted. \end{proof}

\begin{proposition}
Let $K$ be a global field, $L$ a $K$-algebra, $\alge$ a $K$-CSA, and $\phi:L\rightarrow\alge$ a representation.
Assume that, for each irreducible component $\phi_i$, $i=1,\dots,r$, of the representation $\phi:L\rightarrow\alge$, the
representation field $F(\Ha,\phi_i)$ is defined. Then the representation field $F(\Ha,\phi)$ is defined, and in fact
$F(\Ha,\phi)=\bigcap_i F(\Ha,\phi_i)$.
\end{proposition}

\begin{proof}
Choose the orders $\Da$ and $\Da_i$, for $i=1,\dots,r$ as in the preceding Lemma. 
It suffices to prove that $$H(\phi,\Da|\Ha)=\prod_{i=1}^rH(\phi_i,\Da_i|\Ha).$$ One contention follows from the preceding lemma. The other contention follows
from Lemma \ref{l4} if we note that the irreducible representations of the residual algebra $\mathbb{H}_\wp$ of $\phi(\Ha)$ correspond to the irreducible representations of the residual 
algebras of $\phi_i(\Ha)$ for $i=1,\dots,r$ by the proof of Lemma \ref{l4b}. The result follows. 
 \end{proof}
Theorem \ref{uno} and  Theorem \ref{t2} follow from  next corollary:
\begin{cor}
Let $K$ be a global field, let $\alge$ be a $K$-CSA, and let $S$ be a  non-empty finite set
of places in $K$ containing the archimedean places if any.
Let $\Ha\subseteq\alge$ be an $S$-order.
If every irreducible component $\psi:L\rightarrow\alge'$ of the representation $\phi:L\rightarrow\alge$,
where $L=K\Ha$, satisfies any of the following conditions:
\begin{enumerate}
\item $\psi(L)$ is contained in a quaternion algebra,
\item $\psi(L)$ is commutative,
\end{enumerate}then the spinor class field is defined.\end{cor}
\begin{proof}
It suffices to prove that either hypotheses implies the last condition in Lemma \ref{l4}. This follows as in the proof of
\cite[Proposition 4.3]{abelianos} when  $\psi(L)$ is commutative. When $\psi(L)$ is contained in a quaternion algebra,
this follows from Lemma \ref{l4b}, an the example following it.
 \end{proof}
\section{Proof of Theorem \ref{t3}.}

Let $k=K_\wp$ be a local field, let $E/k$ be the unique unramified quadratic extension, and let 
$L=\matrici_n(E)$. Note that there exists, up to change of basis, a unique faithful representation
$\phi:L\rightarrow\matrici_{2n}(k)$ and it can be realized by identifying $E^n$ with $k^{2n}$. Note that
the basis can be chosen in a way that $\oink_E^n$ is identified with $\oink_k^{2n}$. In this case we say that
the representation is integral. Next result is now immediate:
\begin{lemma}\label{ea1}
Let $E/k$ be an unramified quadratic extension of  local fields.
If $\phi:\matrici_n(E)\rightarrow\matrici_{2n}(k)$ is a faithful integral representation, 
for any vector $v$ in $\oink_k^{2n}\backslash\pi\oink_k^{2n}$ we have $\phi\Big(\matrici_n(\oink_E)\Big)v=\oink_k^{2n}$.
\end{lemma}
 Next we consider the order
\begin{equation}\label{mord}
\Ha=\bbmatrix{\oink_k 1_E}{\oink_E}{0}{\oink_E}+\pi\bbmatrix{\oink_E}{\oink_E}{\oink_E}{\oink_E}
\subseteq\matrici_2(E),
\end{equation}
 where $\pi$ is a uniformizing parameter of $k$.
\begin{lemma}
Let $\Ha$ be as in (\ref{mord}), and let $\phi:L\rightarrow\matrici_4(k)$ be a faithful integral representation. Let $\Da=\matrici_4(\oink_k)$.
 Then the relative spinor image is
$$H_\wp(\phi,\Da|\Ha)=k^{*4}\oink_k^*\cup\Big( \pi^2 k^{*4}\oink_k^*\Big)\cup\Big( \pi^3 k^{*4}\oink_k^*\Big).$$
\end{lemma}

\begin{proof}
 To simplify notations, we identify $\Ha$ with $\phi(\Ha)$.
 Note that $a\Ha a^{-1}\subseteq \Da$
if and only if  $\Ha \subseteq a^{-1}\Da a$. Note that $d\in\Da$ and $d'\in\Da'=a^{-1}\Da a$ are equivalent to
$$d\oink_k^4=\oink_k^4,
\qquad d' a^{-1}\oink_k^4=a^{-1}\oink_k^4.$$ 
It follows that $H_\wp(\phi,\Da|\Ha)$ is the set of norms of elements $a$ for which $a^{-1}\oink_k^4$
 is invariant under $\phi(\Ha)$. Let $M$  be a lattice that is invariant under  $\Ha$.
Multiplying by a power of $\pi$ if needed,
we can assume that $M$ is contained in $\oink_K^4$, but not  $\pi\oink_K^4$.
Note that $M$ is also invariant under every sub-order or ideal in $\Ha$:
\begin{enumerate}
\item Since $M$ is invariant under $\pi\matrici_2(\oink_E)$,  a direct application of Lemma \ref{ea1}
shows that  $M$ contains  $\pi\oink_k^4$.
\item The order $\Ha=\bbmatrix{\oink_kI_2}{0}{0}{\oink_E}$, where $I_2$ is the identity matrix,
 can only stabilize lattices of the form $\Lambda_1\times\Lambda_2$ with $\Lambda_2=\pi^r\oink_k^2$,
as follows from Lemma \ref{ea1} for $n=1$.
\item  The order $\Ha=\bbmatrix{\oink_kI_2}{\oink_E}{0}{\oink_kI_2}$ can only stabilize lattices of the form $\Lambda_1\times\Lambda_2$ with $\oink_E\Lambda_2\subseteq\Lambda_1$.
\end{enumerate}

We define the local distance between $\Da_1$ and $\Da_2=b\Da_1b^{-1}$ by
$$\rho_\wp(\Da_1,\Da_2)=v_\wp\big(N(b)\big)+4\enteri\in\enteri/4\enteri,$$
where $N:\matrici_4(k)^*\rightarrow k^*$ is the determinant.  Note that $\rho_\wp$ is well defined, since the 
conjugation stabilizer of $\Da_1$ is $k^*\Da_1^*$, and its set of norms is $k^{*4}\oink_k^*$. From what precedes, any
$\Ha$-invariant  lattice $M=\Lambda_1\times\Lambda_2$, with $\Lambda_1$ and  $\Lambda_2$ of rank $2$, 
and $\pi\oink_k^4\subsetneq M \subseteq\oink_k^4$, is of one of the following types:

\begin{enumerate}
\item If $\Lambda_2\neq\pi\oink_k^2$, then necessarily $\Lambda_2=\oink_k^2$, and in that case $\Lambda_1=\oink_k^2$,
so that $M=\oink_k^4$.
\item If $\Lambda_2=\pi\oink_k^2$, then $\Lambda_1$ is an arbitrary lattice satisfying  $\pi\oink_k^2\subseteq\Lambda_1\subseteq\oink_k^2$. There are three subcases:
\begin{enumerate}
\item If $\Lambda_1=\pi\oink_k^2$ we have $\rho_k(\Da,\Da')=\bar0$.
\item  If $\Lambda_1=\oink_k^2$ we have $\rho_k(\Da,\Da')=\bar2$.
\item  If $\Lambda_1=\oink_k v+\pi\oink_k^2$, for some $v\in\oink_k^2\backslash\pi\oink_k^2$,
 we have $\rho_k(\Da,\Da')=\bar3$.
\end{enumerate}
\end{enumerate}
The result follows.\end{proof}

\subparagraph{Proof of  Theorem \ref{t3}.}
Let $H$ be the Hilbert class field of $K$. Let $\Sigma$ be the spinor class field of maximal orders in $\matrici_4(K)$.
Then $\Sigma$ is the maximal sub-extension of $H$ of exponent $4$, and in particular, the Galois group
$\mathrm{Gal}(\Sigma/K)$ has an element $\sigma$ of order $4$. 
Let $\wp$ be a place satisfying $|[I_\wp,\Sigma/K]|=\sigma$, where $I_\wp$ is the maximal ideal corresponding to $\wp$, and 
$I\mapsto |[I,\Sigma/K]|$ denotes the artin map on ideals.
Let $F'$ be a degree-4 un-ramified cyclic extension such that $f_\wp(F'/K)=4$, and let $F/K$ be the unique quadratic sub-extension.
 We let $\Ha_\wp$ be defined as in equation (\ref{mord}), while we let $\Ha$ be maximal in $\matrici_2(F)$ at all other places. It is immediate from Lemma \ref{ea1} that any maximal order
of $\matrici_2(F)$ is contained in a unique maximal order of $\matrici_4(K)$ at inert places $\wp'\neq\wp$
for $F/K$, whence 
$$H_{\wp'}(\phi,\Da|\Ha)=K^{*4}\oink_K^*,$$ at those places. On the other hand, at places $\wp''\neq\wp$
 splitting $F/K$, every invariant lattice has the form $\pi_{\wp''}^t\oink_{\wp''}^2\times\pi_{\wp''}^s\oink_{\wp''}^2$,
whence $$H_{\wp''}(\phi,\Da|\Ha)=K^{*2}\oink_K^*.$$
 As $f_{\wp''}(F'/K)\leq2$ at the latter places, the Hilbert symbol of every element in 
$H_{\tilde\wp}(\phi,\Da|\Ha)$ is trivial on $F'$ for every place $\tilde\wp\neq \wp$. 
Note that the image in $\mathrm{Gal}(F'/K)$ of $H(\phi,\Da|\Ha)$ is the set
$$\{\mathrm{id},[\wp,F'/K]^2,[\wp,F'/K]^3\}.$$
We conclude that the upper representation field $F^-(\phi,\Ha)$
contains $F'$, while the lower representation field $F_-(\phi,\Ha)$ intersects $F'$ trivially. The result follows.
\qed

\begin{remark}
Assume we have a rank-8 order $ \Ha$, and a representation $\phi$.
 for which  the representation field is not defined. Then, there exists a ffinite place $\wp$ such that $H_\wp(\phi,\Da|\Ha)$
fails to be a group. By the proof of Theorem \ref{uno}, we can assume that the representation $\phi$ is irreducible (and faithfull). Furthermore, $\Ha$ cannot be commutative.
  We conclude that $L=K\Ha$ is a quaternion algebra over a quadratic extension $E/K$. 
\end{remark}

\section{Applications and examples}
We constructed  in \cite{eichler}  an order in $\matrici_n(K)$ , for $n\geq3$,
 whose representation field is not defined\footnote{The fact that the algebra is globally split is not essential here. We have only assumed it for simplicity.} (see also \cite[Ex.3.6]{abelianos}).
We choose $\Ha_\wp$ as the pre-image in $\matrici_n(\oink_K)$ of the residual algebra   $$\mathbb{H}_\wp=\bbmatrix {\mathbb{K}}{\mathbb{K}}0{\matrici_{n-1}(\mathbb{K})},\qquad \mathbb{K}=\oink_K/I_\wp,$$
where $I_\wp$ is the global prime ideal corresponding to $\wp$,  and we showed that the residual algebra $\mathbb{H}_\wp$ alone is not sufficient to determine wether the spinor image $H_\wp(\phi,\Da|\Ha)$ is a group.  In fact, here we can
prove a stronger statement:
\begin{proposition}
Let $k$ be a local field with ring integers $\oink_k$ and maximal ideal $m_k$. Let $E$ be a division $k$-CSA, and let $\Da=\matrici_f(\oink_E)
\subseteq \matrici_f(E)$. Let $\mathbb{E}=\oink_E/m_E$ where $m_E$ is the unique maximal bilateral ideal of $\oink_E$. 
For every residual algebra  $\mathbb{H}\subseteq\matrici_f(\mathbb{E}_\wp)$, there exists a local order $\Ha\subseteq\Da$ of
maximal rank whose image in $\matrici_f(\mathbb{E})$ is $\mathbb{H}$ and whose relative local spinor image
$H_k(\phi,\Da|\Ha)$ is a group.
\end{proposition}

\begin{proof}
Let $M_1\subseteq M_2\subseteq\cdots\subseteq M_r$ be a maximal flag of $(\mathbb{H},\mathbb{E})$-bimodules in $\mathbb{E}^f$, and choose a basis such that $\mathbb{H}$ is contained in the ring of matrices of $\matrici_f(\mathbb{E})$ with a block decomposition of the form (\ref{blocks}). Lift $\mathbb{H}$ to a local order $\Ha'$ which also consists only on matrices
with a block decomposition of the form (\ref{blocks}) for a suitable basis, so in particular, whenever $u_i$ is a generator for
$\Da_i|\Ha_i$, then the matrix in (\ref{blocks2}) is a generator for $\Da|\Ha'$
as soon as each difference $t_i-t_{i-1}$ is chosen big enough. Note
also that, for every irreducible representation $\phi_i$ corresponding to this flag, the residual algebra
$\phi_i(\Ha)/m_k\phi_i(\Ha)$ has a unique representation,
hence the representation field $F(\Ha')$ is defined.
 Now we choose a finite number of generators for
$\Da|\Ha'$ whose reduced norms form a set of representatives for  $H_k(\phi,\Da|\Ha')/k^{*2}$, and note that they are also generators for $\Da|\Ha$,
 where $\Ha=\Ha'+\pi^M\Da$, for a uniformizing parameter $\pi$, if  $M$ is chosen big enough. It follows that, if $\Gamma_k(\phi,\Ha)$ is the group generated by $H_k(\phi,\Da|\Ha)$ and $\Gamma_k(\phi,\Ha')$ is defined analogously, then 
$$\Gamma_k(\phi,\Ha)=\Gamma_k(\phi,\Ha')=H_k(\phi,\Da|\Ha')\subseteq H_k(\phi,\Da|\Ha)\subseteq\Gamma_k(\phi,\Ha).$$
The result follows.\end{proof}

\begin{example}
 Let $\Ha$ and $\Ha'$ be global orders in $\matrici_n(K)$ of the form
$$\Ha=\left(\begin{array}{cccc}\Ha_1&M_{12}&\cdots&M_{1r}\\0&\Ha_2&\cdots&M_{2r}\\
\vdots&\vdots&\ddots&\vdots\\0&0&\cdots&\Ha_r\end{array}\right),\qquad
\Ha'=\left(\begin{array}{cccc}\Ha_1&0&\cdots&0\\0&\Ha_2&\cdots&0\\
\vdots&\vdots&\ddots&\vdots\\0&0&\cdots&\Ha_r\end{array}\right).$$
The results in this work show that whenever $\Ha'$ embeds into every maximal order, so
does $\Ha$. In fact, using strong approximation, it is easy to construct a sequence
of global conjugates of $\Ha$ whose adelization converges in the Hausdorff topology to
$\Ha'$, whence a similar result holds for any genus of orders of maximal rank. This fails to hold in the projective case
(see the remark at the end of \S\ref{rmk1} and the following example).
\end{example}

\begin{example}
Let $\Ha$ be the order
$$\Ha=\bbmatrix{\oink_X}{\mathcal{I}}0{\oink_X},$$ for an arbitrary
ideal (i.e., a one dimensional lattice in $K$). Then $\Ha$ is
contained in an order of every spinor genus of maximal orders in
$\matrici_2(K)$, as in the preceding example. Note that when $X$ is a projective curve over a
finite field $\finitum$, there exists conjugacy classes of maximal
$X$-orders that fail to contain a copy of $\Ha$. For example, if
$\mathcal{I}=\mathcal{L}_D$ is the sheave defined by
$$\mathcal{L}_D(U)=\{f\in K|\mathrm{div}(f)\geq-D\}$$ for a
divisor $D$ on $X$, the ring of global sections $\Ha(X)$ is the
set of all matrices of the form $\bbmatrix af0b$ where $a$ and $b$
are constants and $f\in\mathcal{L}_D(X)$. The dimension of
$\mathcal{L}_D(X)$ grows with the degree of $D$ acording to
Riemann-Roch's Theorem, whence by chosing a divisor $D$ of sufficiently large degree,
we can assume that $\Ha$ cannot be embedded in any order of an arbitrary prescribed finite family.
\end{example}

\end{document}